\documentclass[11pt]{amsart}
	\usepackage{amssymb, latexsym, amsmath, amsfonts}
	\usepackage[pdftex]{graphicx}
	\usepackage{float}
	\usepackage{hyperref}
	\newtheorem{thm}{Theorem}[section]
	\newtheorem{cor}[thm]{Corollary}
	\newtheorem{lem}[thm]{Lemma}
	\newtheorem{prop}[thm]{Proposition}
	\theoremstyle{definition}
	\newtheorem{defn}[thm]{Definition}
	\theoremstyle{remark}
	\newtheorem{rem}[thm]{Remark}
	\numberwithin{equation}{section}
	\newtheorem{ex}[thm]{Example}
	\numberwithin{equation}{section}

	\newcommand{\mbb}{\mathbb}
	\newcommand{\ra}{\rightarrow}

	\newcommand{\ep}{\epsilon}
	\newcommand{\no}{\noindent}

	\newcommand{\De}{\Delta}
	\newcommand{\cal}{\mathcal}
	
	\newcommand{\la}{\lambda}

\begin{document}
	\title{Dynamics of semigroups of entire maps of $\mbb C^k$}
	\thanks{$^1$ Supported by CSIR-UGC(India) fellowship.}
	
	\author{Sayani Bera  and Ratna Pal$^{1}$} 
	\address{Sayani Bera: Department of Mathematics, Indian Institute of Science,
	Bangalore 560 012, India}
	\email{sayani@math.iisc.ernet.in}
	
	\address{Ratna Pal: Department of Mathematics, Indian Institute of Science,
	Bangalore 560 012, India}
	\email{ratna10@math.iisc.ernet.in}
	
	\pagestyle{plain}
	
	\begin{abstract}
	 The goal of this paper is to study some basic properties of the Fatou and Julia sets for a family of holomorphic endomorphisms of $\mbb C^k,\; k \ge 2.$ We are particularly interested in studying these sets for semigroups generated by various classes of holomorphic endomorphisms of $\mbb C^k,\; k\ge 2.$ We prove that if the Julia set of a semigroup $G$ which is generated by endomorphisms of maximal generic rank $k$ in $\mbb C^k$ contains an isolated point, then $G$ must contain an element that is conjugate to an upper triangular automorphism of $\mbb C^k.$ This generalizes a theorem of Fornaess--Sibony. Secondly, we define recurrent domains for semigroups and provide a description of such domains under some conditions. 
	\end{abstract}
	\maketitle
	\section{Introduction}\label{Introduction}
	The purpose of this note is to study the Fatou--Julia dichotomy, not for the iterates of a single holomorphic endomorphism of $\mbb C^k, \; k \ge 2$, but for a family $\mathcal{F}$ of such maps. The Fatou set of $\mathcal{F}$ will be by definition the largest open set where the family is normal, i.e., given any sequence in $\mathcal{F}$ there exists a subsequence which is uniformly convergent or divergent on all compact subsets of the Fatou set, while the Julia set of $\mathcal{F}$ will be its complement.
	\par\medskip\no
	We are particularly interested in studying the dynamics of families that are semigroups generated by various classes of holomorphic endomorphisms of $\mbb C^k,\; k \ge 2.$ For a collection $\{\psi_{\alpha}\}$ of such maps let
	\[G=\langle \psi_{\alpha}\rangle\]
	denote the semigroup generated by them. The index set to which $\alpha$ belongs is allowed to be uncountably infinite in general. The Fatou set and Julia set  of this semigroup G will be henceforth denoted by $F(G)$ and $J(G)$ respectively. Also for  a holomorphic endomorphism  $\phi$ of $\mbb C^k ,$ $F(\phi)$ and $J(\phi)$, will denote the Fatou set and Julia set  for the family of iterations of $\phi.$ The $\psi_{\alpha}$'s that will be considered in the sequel will belong to one of the following classes:
		\begin{itemize}
	\item $\mathcal{E}_k:$ The set of holomorphic endomorphisms of $\mbb C^k$ which have maximal generic rank $k.$
	\item $\mathcal{I}_k:$  The set of injective holomorphic endomorphisms of $\mbb C^k.$
	\item $\mathcal{V}_k:$ The set of volume preserving biholomorphisms of $\mbb C^k.$
	\item $\mathcal{P}_k:$ The set of proper holomorphic endomorphisms of $\mbb C^k.$
		\end{itemize}
	The main motivation for studying the dynamics of semigroups in higher dimensions comes from the results of Hinkkanen--Martin\cite{HM1} and Fornaess--Sibony \cite{FS}. While \cite{HM1} considers the dynamics of semigroups generated by rational functions on the Riemann sphere, \cite{FS} puts forth several basic results about the dynamics of the iterates of a single holomorphic endomorphism of $\mbb C^k,\; k \ge 2.$ Under such circumstances, it seemed natural to us to study the dynamics of semigroups in higher dimensions.
	\par\medskip\no
	Section 2 deals with basic properties of $F(G)$ and $J(G)$ when $G$ is generated by elements that belong to $\mathcal{E}_k$ and $\mathcal{P}_k.$ The main theorem in Section 3 states that if $J(G)$ contains an isolated point, then $G$ must contain an element that is conjugate to an upper triangular automorphism of $\mbb C^k.$ Finally we define recurrent domains for semigroups in Section 4 and provide a classification of such domains under some conditions which are generalizations of the corresponding statements of Fornaess--Sibony \cite{FS} for the iterates of a single holomorphic endomorphism of $\mbb C^k,\; k\ge  2.$ The classification for recurrent Fatou components for the iterates of holomorphic endomorphisms of  $\mbb P^2$ and $\mbb P^k$ is studied in \cite{FS2} and \cite{Fornaess-Rong} respectively. In \cite{FS2} Fornaess--Sibony also gave a classification of recurrent Fatou components for iterations of H\'{e}non maps inside $K^+$, which was initially considered by Bedford--Smillie in \cite{BS2}. A  classification for non-recurrent, non-wandering  Fatou components of $\mbb P^2$ is given  in  \cite{Weickert}, whereas  a classificiation of invariant Fatou components for nearly dissipative H\'{e}non maps is studied in \cite{Peters-Lyubich}. 
	
	\medskip
	
	\no {\bf Acknowledgement:} We would like to thank Kaushal Verma for valuable discussions and comments.
	
		\section{Properties of the Fatou set and Julia set for a semigroup G}\label{section 1}
		In this section we will prove some basic properties of the Fatou set and the Julia set for semigroups. 	
		\begin{prop}\label{lem1}
			Let $G$ be a semigroup generated by elements of $\mathcal{E}_k$ where $k \ge 2$ and for any $\phi \in G$ define 
			\[\Sigma_{\phi}=\{z \in \mbb C^k: \det{\phi(z)}=0\}.\] Then for every $\phi \in G$
			\begin{enumerate}
				\item[(i)] $\phi(F(G)\setminus \Sigma_{\phi})\subset F(G).$ 
				\item[(ii)] $J(G)\cap \phi(\mbb C^k)\subset \phi(J(G))$, if $G$ is generated by elements of $\mathcal{P}_k$ or $\mathcal{I}_k.$
			\end{enumerate}
			
		\end{prop}
		
		\begin{proof}
			Note that $\phi \in G$ is an open map at any point $z \in F(G)\setminus \Sigma_{\phi}.$ Since for any sequence $\psi_n \in G$, the sequence $\psi_n \circ \phi$ has a convergent subsequence around a neighbourhood of $z$ (say $V_z$), $\psi_n$ also has a convergent subsequence on the open set $\phi(V_z)$ containing $\phi(z).$ 
			
			\smallskip
			
			\no Now if $G$ is generated by elements of $\mathcal{P}_k$ or $\mathcal{I}_k$ then $\phi$ is an open map at every point in $\mbb C^k.$ Then the Fatou set is forward invariant and hence the Julia set is backward invariant in the range of $\phi.$
		\end{proof}
		
	
	\no A family of endomorphisms $\mathcal{F}$ in $\mbb C^k$ is said to be locally uniformly bounded on an open set $\Omega \subset \mbb C^k$ if for every point there exists a small enough neighbourhood of the point (say $V\subset \Omega$) such that $\mathcal{F}$ restricted to $V$ is bounded i.e.,
	\[\|f\|_{V}=\sup_{V}|f(z)|<M\]for some $M>0$ and for every $f \in \mathcal{F}.$
		
		\begin{prop} \label{bound}
			Let $G=\langle\phi_1,\phi_2,\hdots,\phi_n\rangle$, where each $\phi_j \in \mathcal{E}_k$ and let $\Omega_G$ be a Fatou component  of $G$ such that $G$ is locally uniformly bounded on $\Omega_G.$ Then for every $\phi \in G$ the image of $\Omega_G$ under $\phi$ i.e., $\phi(\Omega_G)$ is contained in Fatou set of $G.$
		
		\end{prop}
		
		\begin{proof}
			
			Let $K \subset\subset \Omega_G$, i.e., $K$ is a relatively compact subset of $\Omega_G,$ then 
			
			\smallskip
			
			\no {\it Claim:-}  $\Omega_G$ is a Runge domain i.e., $\hat{K} \subset \Omega_G$ where 
			\[\hat{K}:=\{z \in \mbb C^k: |P(z)| \le \sup_{K}|P| ~\mbox{for every  polynomial}~ P\}.\]
			Let $K_{\delta}=\{z \in \mbb C^k: ~\mbox{dist}(z,K)\le \delta\}$. Choose $\delta >0$ such that $K_{\delta} \subset\subset \Omega_G.$ Now note that $\hat{K_\delta}\subset \subset \mbb C^k$ , $\hat{K_{\delta}}\supset \hat{K}$ and $G$ is uniformly bounded on $K_{\delta}.$ Pick $\phi \in G.$ Then there exists a polynomial endomorphism $P_{\phi}$ of $\mbb C^k$ such that 
			\begin{align*}
			|\phi(z)-P_{\phi}(z)| \le \ep ~\mbox{for every}~z \in \hat{K_{\delta}}\\
			\mbox{i.e.,} \hspace{1cm}|P_{\phi}(z)|-\ep \le |\phi(z)| \le |P_{\phi}(z)|+\ep .
			\end{align*}
			Hence
			\begin{align*}
			|\phi(z)| &\le |P_{\phi}(z)|+\ep \le \sup_{K_\delta}|P_{\phi}(z)|+\ep \\
			&\le \sup_{K_\delta}|{\phi}(z)|+2 \ep \le M+2\ep
			\end{align*}
			for every $z \in \hat{K_{\delta}}$ and some constant $M > 0.$ So $G$ is uniformly bounded on $\hat{K_{\delta}}$ and $\hat{K} \subset \Omega_G.$ 
			
			\smallskip
			
			\no Let $$\Sigma_i=\{z \in \mbb C^k: \det{\phi_i(z)}=0\}$$ for every $1 \le i \le n$ and $$\Sigma=\bigcup_{i=1}^n \Sigma_i.$$ Thus $\phi_i $ for every $i$, where $1 \le i \le n$ is an open map in $\Omega_G\setminus \Sigma$. Hence $\phi_i(\Omega_G\setminus\Sigma)$ is contained inside a Fatou component say $\Omega_i$ and $G$ is locally uniformly bounded on each of $\Omega_i$ for every $1 \le i \le n$ i.e., each $\Omega_i$ is a Runge domain. 
			
			\smallskip
			
			\no Now pick $p \in \Omega_G \cap \Sigma.$ Since $\Sigma$ is a set with empty interior, there exists a sufficiently small disc centered at $p$ say $\Delta_p$ such that $\overline{\De}_p\setminus \{p\} \subset \Omega_G \setminus \Sigma.$ Then $\phi_i(\overline{\De}_p\setminus \{p\}) \subset \Omega_i$ for every $1\le i \le n$ and since each $\Omega_i$ is Runge $\phi_i(p) \in \Omega_i$ i.e., $\phi_i(\Omega_G)$ is contained in the Fatou set for every $1 \le i \le n.$ Now for any $\phi \in G$ there exists a $m > 0$ such that
			\[\phi=\phi_{n_1}\circ \phi_{n_2}\circ\hdots\circ\phi_{n_m}\] where
			$1 \le n_j \le n$ for every $1 \le j \le m.$ Thus applying the above argument repeatedly for each $\phi_{n_j}(\tilde{\Omega}_j)$ where $G$ is locally uniformly bounded on $\tilde{\Omega}_j$ it follows that $\phi(\Omega_G)$ is contained in the Fatou set of $G$.	
		\end{proof}
		
		\begin{prop}
			If $G=\langle\phi_1,\phi_2,\hdots,\phi_n\rangle$ where each $\phi_i \in \mathcal{E}_k$  for every $1 \le i \le n$ and let $\Omega_G$ be a Fatou component of $G$. Then for any $\phi \in G$ there exists a Fatou component of $G$, say $\Omega_{\phi}$ such that $\phi(\Omega_G) \subset \bar{\Omega}_{\phi}$ and 
			\[\partial\Omega_G\subset \bigcup_{i=1}^n \phi_i^{-1}(\partial\Omega_{\phi_i}).\] 	
		\end{prop}
		
		\begin{proof}
			
		Let $\phi \in G$ and let $\Sigma_{\phi}$ denote the set of points in $\mbb C^k$ where the Jacobian of $\phi$ vanishes. Since $\Omega_G \setminus \Sigma_{\phi}$ is connected it follows that $\phi(\Omega_G \setminus \Sigma_{\phi}) \subset \Omega_{\phi}$ where $\Omega_{\phi}$ is a Fatou component of $G$ and by continuity $\phi(\Omega_G) \subset \bar{\Omega}_{\phi}.$
		
		\smallskip
	
		\no Pick $p \in \partial \Omega_G$ such that $p \notin \partial\Omega_{\phi_i}$ for every $1 \le i \le n.$ Since $\phi_i(\Omega_G) \subset \bar{\Omega}_{\phi_i}$ , $\phi_i(p) \in \Omega_{\phi_i}$ for every $1 \le i \le n.$ So there exists  $V_{\phi_i}$ an open neighbourhood of $\phi_i(p)$ in $\Omega_{\phi_i}$ for every $i.$ Let $V_p$ be a neighbourhood of $p$ such that
		\[\bar{V}_p\subset\bigcap_{i=1}^n \phi_i^{-1}(V_{\phi_i}).\]
		 Let $\{\psi_n\}$ be a sequence in $G$ and without loss of generality it can be assumed that there exists a subsequence such that $\psi_n=f_n \circ \phi_1.$ Now $\phi_1(\bar{V}_p)$ is a compact subset in $\Omega_1$ and $f_n$ has a subsequence which either converges uniformly on $\phi_1(\bar{V}_p)$ or diverges to infinity. Thus $V_p$ is contained in the Fatou set of $G$ which is a contradiction!
		\end{proof}
		\no 
		The next observation is an extension of the fact that if $\phi \in \mathcal{P}_k$, then $F(\phi)=F(\phi^n)$ for every $n > 0$ for the case of semigroups.
		\begin{defn}
		Let $G$ be a semigroup generated by endomorphisms of $\mbb C^k.$ A sub semigroup $H$ of $G$ is said to have finite index if there is a finite collection of elements say $\psi_1,\psi_2,\hdots,\psi_
		{m-1} \in G $ such that
		\[G =\Big(\bigcup_{i=1}^{m-1} \psi_i \circ H\Big) \cup H.\]
		 The index of $H$ in $G$ is the smallest possible number $m.$
		\end{defn}
		\begin{defn}
		A sub semigroup $H$ of a semigroup $G$ of endomorphisms of $\mbb C^k$ is of co--finite index if there is a finite collection of elements say $\psi_1,\psi_2,\hdots,\psi_{m-1} \in G $ such that either \[\psi \circ \psi_{j} \in H \; \mbox{or}\; \psi \in H\]
		for every $\psi \in G$ and for some $1 \le j \le m-1.$ 
		The index of $H$ in $G$ is the smallest possible number $m.$
		\end{defn}
		\begin{prop}\label{index}
		Let $G$ be a semigroup generated by proper holomorphic endomorphisms of $\mbb C^k$ and $H$ be a sub semigroup of $G$ which has a finite (or co--finite) index in $G.$ Then $F(G)=F(H)$ and $J(G)=J(H).$
		\end{prop}
		\begin{proof}
		From the definition itself it follows that $F(G) \subset F(H).$ To prove the other inclusion, pick any sequence $\{\phi_n \}\in G$. Since  $H$ has a finite index  in $G$, there exists $\psi_i$, $1 \le i \le m-1$ such that 
			\[G =\Big(\bigcup_{i=1}^{m-1} \psi_i \circ H\Big) \cup H.\]
			So without loss of generality one can assume that there exists a subsequence say $\phi_{n_k}$ with the property \[\phi_{n_k}=\psi_1 \circ h_{n_k}\] where $\{h_{n_k}\}$ is a sequence in $H.$ Now on $F(H)$, the sequence $\{h_{n_k}\}$ has a convergent subsequence. Hence, so do $\{\phi_{n_k}\}$ and $\{\phi_n\}$ as $\psi_1$ is a proper map in $\mbb C^k.$
		\end{proof}
		\no 
		Let $G$ be a semigroup $$G=\langle \phi_1,\phi_2,\hdots,\phi_m\rangle$$ where $\phi_i \in \mathcal{P}_k$, for every $1 \le i \le m$ and each of these $\phi_i$'s commute with each other, i.e., $\phi_i \circ \phi_j=\phi_j \circ \phi_i$ for $i \neq j.$ Let $H$ be a sub semigroup of $G$ defined as
		\[H=\langle \phi_1^{l_1},\phi_2^{l_2},\hdots ,\phi_m^{l_m}\rangle\]
		where $l_i>0$ for every $1 \le i \le m.$ Then $H$ has a finite index in $G$ and hence by Proposition \ref{index} $F(G)=F(H).$
		\begin{cor}
		Let $\phi_i$ be elements in $\mathcal{P}_k$ for $1 \le i \le m$, $l=(l_1,l_2,\hdots,l_m)$ a $m-$tuple of positive integers and $G_l=\langle\phi_1^{l_1},\phi_2^{l_2},\hdots, \phi_m^{l_m}\rangle.$ Then $F(G_l)$ and $J(G_l)$ are independent of the $m-$tuple $l$, if $\phi_i \circ \phi_j=\phi_j\circ\phi_i$ for every $1 \le i,j\le m,$ i.e., given two $m-$tuples $p$ and $q$, $F(G_p)=F(G_q).$
		\end{cor}
		\begin{proof}
		Since $G_l$ has a finite index in $G$ for every $m-$tuple $l=(l_1,l_2,\hdots,l_m)$, it follows that $F(G_l)=F(G)$ and $J(G_l)=J(G).$
		\end{proof}
		\begin{ex}
	    Let $G=\langle f,g \rangle$ where $f(z_1,z_2)=(z_1^2,z_2^2)$ and $g(z_1,z_2)=(z_1^2/a,z_2^2)$ where $a \in \mbb C$ such that $|a|>1.$ Then it is easy to check that \[J(f)=\big\{|z_1|=1\big\} \times\big \{|z_2| \le 1 \big \} \cup \big\{|z_1|\le 1\big\}\times \big\{|z_2| =1\big\}\] and
	    \[J(g)=\big\{|z_1|=|a|\big\} \times \big\{|z_2| \le 1\big\} \cup \big\{|z_1|\le |a|\big\}\times \big\{|z_2| =1\big\}.\]
	    Now consider the bidisc $\{|z_1 | < 1, |z_2 | < 1\}.$ Clearly this domain is forward invariant under
	    both $f$ and $g$. This shows that $\{|z_1 | < 1, |z_2 | < 1\} \subset F(G)$. Similarly observe that
	    $$\{|z_2 | > 1\} \cup \{|z_1 | > |a|\} \subset F(G).$$ We claim that
	    \[\big\{ 1 \le |z_1| \le |a|\big\} \times \big\{ |z_2| \le 1 \big\} \subset J(G). \]
	    Note that $\{|z_1|=|a|,|z_2| \le 1\}$ is contained inside $J(G)$ and since $J(G)$ is backward invariant it follows that \[\{|z_1|=|a|^{1/2},|z_2| \le 1\} \subset f^{-1}(\{|z_1|=|a|,|z_2| \le 1\}) \subset J(G).\] So inductively we get that \[\{|z_1|=|a|^t,|z_2| \le 1\} \subset J(G)\] for any $t=k2^{-n}$ where $1 \le k \le 2^n$ and $n \ge 1.$ As $\{k2^{-n}: 1 \le k \le 2^n,\; n \ge 1\}$ is dense in $[0,1]$, it follows that $\{ 1 \le |z_1| \le |a|\} \times \{ |z_2| \le 1 \} \subset J(G).$ Thus the Julia set of the semigroup $G$ is not forward invariant and clearly from the above observations one can prove that 
	    \[J(G)=\big\{ |z_1| \le 1 \big\} \times \big\{ |z_2| = 1 \big\}\cup\big\{ 1 \le |z_1| \le |a|\big\} \times \big\{ |z_2| \le 1 \big\}.\]
	    \end{ex}
	    
		\begin{ex}
		Let $T_0 (z) = 1,\; T_1 (z) = z$ and $T_{n+1}(z) = 2z T_n (z)-T_{n-1}(z)$ for
		$n \ge 1$ and $G =\langle f_0 , f_1 , f_2 , \hdots  \rangle$, with $f_i (z_1 , z_2 ) = (T_i (z_1 ), z_2^2 )$ for $i \ge 0.$ Consider
		\[
		G_1 =\langle T_0 (z_1 ), T_1 (z_1 ), T_2 (z_1 ), ... \rangle,\; G_2 =\langle z_2^2 \rangle.\]
		Since any sequence in $G_1$ is uniformly unbounded on the complement of $[-1,1]$ it follows that 
		\[J(G) = [-1, 1] \times\{|z_2 | \le 1\}.\]
		Also as $J(G_1)\subset \mbb C$ is completely invariant so is $J(G).$ 
		\end{ex}
		\section{Isolated points in the Julia set of a semigroup G.}
		
		\begin{prop}\label{prop1} 
		Let $G=\langle\phi_1, \phi_2, \hdots \rangle$ where each $\phi_i \in \cal E_k$. If the Julia set $J(G)$ contains an isolated point (say $a$) then there exists a neighbourhood $\Omega_a$ of $a$ such that $\Omega_a\setminus\{a\} \subset F(G)$ and $\psi \in G$ which satisfies  $\Omega_a \subset \subset \psi(\Omega_a).$
		In particular, if $G$ is  a semigroup generated by proper maps, then $\psi^{-1}(a)=a$.
		\end{prop}
		\begin{proof}
		Assume $a=0$ is an isolated point in the Julia set $J(G).$ Then there exists a sufficiently small ball $B(0,\ep)$ around $0$ such that $B(0,\ep)\setminus \{0\}$ is contained $F(G).$ Let $$A:=\{z: \ep/2\le|z|\le\ep\}.$$ Then $A \subset F(G).$ 
		
		\smallskip
		
		\no{\it Claim:} There exists a sequence $\phi_n \in G$ such that $\phi_n$ diverges to infinity on $A.$ 
		
		\smallskip
		
		\no Suppose not. Then for every sequence $\{\phi_n\} \in G$, there exists a subsequence  $\{\phi_{n_k}\}$ which converges to a finite limit in $A.$ By the maximum modulus principle
		\[\|{\phi_{n_k}}\|_{B(0,\ep)}<M.\] 
		By the Arzel\'{a}--Ascoli Theorem it follows that $\phi_{n_k}$ is equicontinuous on $B(0,\ep)$, which contradicts that  $0 \in J(G).$ 
		
		\smallskip
		
		\no By the same reasoning as above there exists a sequence $\{\phi_n\} \in G$ such that it diverges uniformly to infinity on $A$ but does not diverge uniformly to infinity on $B(0,\ep)$, since it would again imply that $B(0,\ep)$ is contained in the Fatou set of $G.$  Thus there exists a sequence of points $x_n$ in $B(0,\ep)$ such that $\phi_n(x_n)$ is bounded  i.e.,$$|\phi_n(x_n)|< M$$ for some large $M>0.$ So we can choose a subsequence of this $\{\phi_n\}$ and relabel it as $\{\phi_n\}$ again such that it satisfies the following condition:
		\[\phi_n(x_n) \rightarrow q ~\mbox{and}~ x_n \rightarrow p \]
		where $p \in \overline{B(0,\ep)}.$ 
		
		\smallskip 
		
		\no {\it Claim:} $p=0$. 
		
		\smallskip
		
		\no Suppose not. Then $\phi_n(p)$ is bounded. Let $\widetilde{A}=\{z: \min(|p|, \ep/2)\le |z| \le \ep \}.$ Then $\widetilde{A} \supseteq A.$ Now $\phi_{n_k}(p)$ converges on $\widetilde{A}$, then $\phi_{n_k}$ on $\widetilde{A}$ converges to a finite limit, and hence on $A$ by the maximum modulus principle. This is a contradiction! 
		
		\smallskip
		 
		\no Since ${\phi_{n}|}_{\partial B(0,\ep)} \rightarrow \infty$ for large $n$ $$\|{\phi_n} \|_{\partial B(0,\ep)} \gg |q|.$$
		
			\no Thus for a sufficiently large $R>0 $ and $n$   $$B(0,|q|+R)\cap \phi_n(B(0,\ep))\neq \emptyset.$$ Now, if $B(0,\ep) \nsubseteq \phi_n(B(0,\ep))$, then $B(0,|q|+R)\nsubseteq \phi_n(B(0,\ep))$ since $B(0,\ep) \subset B(0,|q|+R) $ for large $R>0.$ Then there exists $y_n \in \partial B(0,\ep)$ such that $|\phi_n(y_n)| <|q|+R$, which is not possible. Hence $B(0,\ep) \subset \subset \phi_n(B(0,\ep))$ for sufficiently large $n.$ Relabel this $\phi_n$ as $\psi$ and consider the neighbourhood $\Omega_0$ as $B(0,\ep).$
		
		\smallskip
		
		\no Since $0 \in B(0,\ep)\subset \psi(B(0,\ep))$, there exists $\alpha \in B(0,\ep)$ such that $\psi(\alpha)=0.$ From Proposition \ref{lem1}  it follows that $\alpha=0.$ 
		\end{proof}
		\begin{thm}\label{iso}
		 Let $G=\langle\phi_1, \phi_2, \hdots \rangle$ where each $\phi_i \in \cal I_k.$ If the Julia set $J(G)$ contains an isolated point, say $a$ then there exists an element $\psi \in G$ such that $\psi$ is conjugate to an upper triangular automorphism.
		\end{thm}

\begin{proof}
Without loss of generality we can assume that $a=0.$ Now by Proposition \ref{prop1} it follows that there exists a sufficiently small ball $B(0,\ep)$ around $0$ and an element $\psi \in G$ such that $B(0,\ep) \subset \subset \psi(B(0,\ep)).$ Since $\psi$ is injective map in $\mbb C^k$, $\psi(B(0,\ep))$ is biholomorphic to $B(0,\ep)$ and hence we can consider the inverse i.e.,\[\psi^{-1}:\psi(B(0,\ep)) \to B(0,\ep).\] Note that $\psi(B(0,\ep))$ is bounded and $B(0,\ep)$ is compactly contained in $\psi(B(0,\ep)).$ Therefore there exists an $\alpha>1$ such that the map defined by
\[ \psi_{\alpha}=\alpha \psi^{-1}(z)\] 
is a self map of the bounded domain $\psi(B(0,\ep))$ with a fixed point at $0.$ Then by the Carath\'{e}odory--Cartan--Kaup--Wu Theorem (See Theorem 11.3.1 in \cite{Krantz}) it follows that all the eigenvalues of $\psi_\alpha$ are contained in the unit disc. Hence $0$ is a repelling fixed point for $\psi$ and also is an isolated point in the Julia set of $\psi.$ 

\smallskip

\no Since $B(0,\ep)\setminus \{0\} \in J(G)$, $B(0,\ep)\setminus \{0\}$ is also contained in the Fatou set of $\psi$ and using the same argument as in the Proposition \ref{prop1} there exists a subsequence (say $n_k$) such that 
\[\|\psi^{n_k}\|_{\partial B(0,\ep)} \to \infty \]
uniformly. Thus for any given $R>0$ there exists $k_0$ large enough such that $B(0,R) \subset \psi^{n_{k_0}}(B(0,\ep)).$ Hence $\psi$ is an automorphism of $\mbb C^k$ and the basin of attraction of $\psi^{-1}$ at $0$ is all of $\mbb C^k.$ Now by the result of Rosay--Rudin (\cite{RR1}) $\psi$ is conjugate to an upper triangular map.
\end{proof}
		\begin{rem}
		The proof here shows that there exists a sequence $\phi_n \in G$ such that each $\phi_n$ is conjugate to an upper triangular map.
		\end{rem}
		\no
		Recall that a domain $\omega$ is holomorphically homotopic to a point in a domain $\Omega$ if there exists a continuous map $h:[0,1]\times \bar{\omega}\to \Omega$ with $h(1,z)=z$ and $h(0,z)=p$ where $p \in \omega$ and $h(t,\centerdot)$ is holomorphic in $\omega$ for every $t \in [0,1].$
		\begin{prop}\label{prop2}
		Let $\phi$ be a non-constant endomorphism of $\mbb C^k$ such that on a bounded domain $U \subset F(\phi)$, the map $\phi$ is proper onto its image, $U \subset \subset \phi(U)$ and $U$ is holomorphically homotopic to a point in $\phi(U)$ then
		\begin{enumerate}
		\item[(i)] $\phi$ has a fixed point, say $p$ in $U$.
		\item[(ii)] $\phi$ is invertible at its fixed points.
		\item[(iii)]The backward orbit of $\phi$ at the fixed point in $U$ is finite i.e, $O^- (p) \cap U$ is finite where
		\[O^-_{\phi} (p)=\{ z \in \mbb C^k: \phi^n(z)=p, n \ge 1\}.\]
		\end{enumerate}
		\end{prop}
		\begin{proof}
		That the map $\phi$ has a fixed point $p$ in $U$ follows from Lemma 4.3 in \cite{FS}.
		\par\smallskip\no 
		Without loss of generality we can assume $p=0$. Consider $\psi(z)=\phi(p+z)-p$ and $\Omega=\{z-p:z \in U \}.$ Then $\psi$ is the required map with the properties $\Omega \subset \subset \psi(\Omega)$ and $0$ is a fixed point for $\psi.$
		\par\smallskip\no 
		Suppose $\psi$ is not invertible at $0,$ i.e., $A=D\psi(0)$ has a zero eigenvalue. Let $\la_i$, $1 \le i \le k$ be the eigenvalues of $A.$ Therefore there exist an $\alpha$ such that $0< \alpha < 1$  and $1 <m \le k$ such that $0=|\la_i|< \alpha$ for $1 \le i \le m$ and $|\la_i|> \alpha$ for $m< i \le k.$ Choose $\delta>0$ such that 
		\[0 < \|D_{\mathbb C}\psi(z)-A\| < \epsilon_0=\min\Big\{\alpha,\big||\la_i|-\alpha\big|\Big\}\]
		for $z \in B(0, \delta)$ and $m< i \le k.$ Let $\Psi$ be a Lipschitz map in $\mbb C^k$ such that 
		\[Lip(\Psi)=\|A\|+\epsilon_0 \]
		and
		\[\Psi \equiv \psi \;\; \mbox{on} \; \; B(0,\delta).\]
		Now \[W_s^{\Psi}:=\{z \in \mbb C^k:|\alpha^n \Psi^n(z)|\;\mbox{is bounded}\;\}\]
		can be realized as a graph of a continuous function (See \cite{Yoccoz}) $G_{\Psi}: \mbb C^m \to \mbb C^{k-m}$ such that $G_{\Psi}(0)=0.$ Since $${W_s^{\Psi}}={W_s^{\psi}} \; \text{on} \; B(0,\delta/2)$$
		$W_s^{\psi} \cap \Omega $ is an infinite non-empty set containing $0.$ Also ${\psi^{n_k}}_{|\bar{\Omega}} \to \psi_0$ for some sequence $n_k$ and $\psi_0$ is holomorphic on the component (say $F_0$) of $F(\psi)$ containing $\Omega$. Let \[W_1^{\psi}=\{z \in F_0:\psi^{n_k}(z)\to 0 \;\mbox{as}\; k \to \infty \}.  \]
		Then $ W_s^{\psi} \cap F_0 \subset W_1^{\psi} $ and \[W_1^{\psi}=\bigcap_{i=1}^k {\psi_{0,i}}^{-1}(0)\]
		where ${\psi_{0,i}}$ is the $i-$th coordinate function of $\psi_0.$ If $W_1^{\psi} \cap \partial\Omega=\emptyset $ then $W_1^{\psi} \cap\Omega$ and hence $ W_s^{\psi} \cap \Omega$ will have to be finite which is not true. Thus there exists a positive integer $n_0$ such that $\psi^{n_0}(\partial\Omega) \cap \Omega \neq \emptyset$ but by assumption it follows that $\Omega \subset \subset \psi^n(\Omega)$ for all $n \ge 1,$ i.e., $\psi^{n}(\partial\Omega) \cap \Omega=\emptyset$ for all $n > 0.$ This proves that $A$ has no zero eigenvalues.
		\par\smallskip\no 
		Note that this observation also reveals that $W_1^{\psi} \cap \Omega $ has to be a finite set, and since
		\[O^-_{\psi}(0) \subset W_1^{\psi}\]
	    the backward orbit of $0$ under $\psi$  is finite.
		\end{proof}
		\no 
		Now we can state and prove Theorem \ref{iso} for semigroups generated by the elements of $\mathcal{E}_k.$
			\begin{thm}
				 Let $G=\langle\phi_1, \phi_2, \hdots \rangle$ where each $\phi_i \in \mathcal{E}_k.$  If the Julia set $J(G)$ contains an isolated point (say $a$) then there exists a $\psi \in G$ such that $\psi$ is conjugate to an upper triangular automorphism.
				\end{thm}
		\begin{proof}
		Assume $a=0.$ Then as before by Proposition \ref{prop1} there exists a map $\psi \in G$ and a domain $\Omega$ such that $\Omega \subset \subset \psi(\Omega).$
		\par\smallskip\no 
		If $0$ is in the Julia set of $\psi$ then $0$ is an isolated point in $J(\psi)$ and by applying Theorem 4.2 in \cite{FS}, it follows that $\psi$ is conjugate to an upper triangular automorphism.
		\par\smallskip\no
		Suppose $\Omega \subset F(\psi).$  By Proposition \ref{prop2}, $\psi$ has a fixed point in $\Omega$ i.e.,$\{\psi^n\}$ has a convergent subsequence in $\bar{\Omega}$. 
		\par\smallskip\no 
		{\it Case 1:} Suppose that $G=\langle\phi_1,\phi_2, \hdots\rangle$ where each $\phi_i \in \mathcal P_k.$ 
		\par\smallskip\no 
		Applying Proposition \ref{prop1}, we have that $\psi^{-1}(0)=0$ and there exists $\psi \in G$ such that 
		\begin{align}
		\Omega \subset \subset B(0,R) \subset \subset \psi(\Omega)
		\end{align}
		where $\Omega$ is a sufficiently small ball at $0$ and $R>0$ is a sufficiently large number. Now let $\omega$ is the component of $\psi^{-1}(B(0,R))$ in $\Omega$ containing the origin. Also from Proposition \ref{prop2} it follows that $0$ is a regular point of $\psi$, which implies that $\psi$ is a biholomorphism on $\omega.$ Define $\Psi_\beta$ on $\psi(\omega)$ as \[\Psi_\beta (z)=\beta \psi^{-1}(z)\] and note that $\Psi_\beta$ is a self map of $B(0,R)$ for some $\beta>1$ with a fixed point at $0.$ Then the eigenvalues of $D_{\mbb C}{\Psi_{\beta}}(0)$ are in the closed unit disc, i.e., \[\beta |\la_i^{-1}| \le 1\] where $\la_i$ are eigenvalues of $A.$ Hence $0$ is a repelling fixed point for the map $\psi$ and $0 \notin F(\psi).$ Since $0$ is an isolated point in the Julia set of $\psi$, by Theorem 4.2 in \cite{FS} $\psi$ is conjugate to an upper triangular automorphism of $\mbb C^k.$ 
		\par\smallskip\no 
		{\it Case 2:} Suppose that $G=\langle\phi_1,\phi_2, \hdots\rangle$ where each $\phi_i \in \mathcal E_k.$
		\par\smallskip\no 
		As before by Proposition \ref{prop1} there exists $\psi \in G$ such that $$\Omega \subset B(0,R) \subset \psi(\Omega)$$ and let $\omega$ be a component of $\psi^{-1} (B(0,R)) \subset \Omega.$ Then $\omega$ satisfies all the condition of Proposition \ref{prop2} and hence there exists a fixed point $p$ of $\psi$ in $\omega$ and $O^-_{\psi}(p)\cap \omega$ is finite. 
		\par\smallskip\no 
		{\it Claim:} $\psi^{-1}(p) \cap \omega=p$ 
		\par\smallskip\no 
		Suppose not i.e., \[\#\{\psi^{-1}(p)\}=\text{the cardinality of}\; \psi^{-1}(p) =m\] and $m \ge 2.$ Let $a_1 \in \psi^{-1}(p)\setminus \{p\}$  in $\omega$ and define $$S_1=O^-_{\psi}(a_1) \cap \omega.$$ Then $S_1 \subset O^-_{\psi}(p)\cap \omega.$ Now choose inductively $a_n \in \psi^{-1}(a_{n-1}) \setminus \{a_{n-1}\}$ for $n \ge 2$ and define $$S_n=O^-_{\psi}(a_n) \cap \omega.$$ Then \[S_n \subset S_{n-1} \; \; \mbox{and} \; \; \bigcup_{i=1}^n S_i\subset O^-_{\psi}(p)\cap \omega\] for every $n \ge 2.$ Note that $a_n \notin S_n,$ otherwise there is a positive integer $k_n >0$ such that $\psi^{k_n}(a_n)=a_n$ i.e., $a_n$ is a periodic point of $\psi$, and $$\psi^{k_n+m}(a_n)=p$$ for any $m>n$. Since $O^-_{\psi}(p)\cap \omega$ is finite it follows that $S_n$ has to be empty for large $n$. This implies that there exists a $n_0 \ge 1$ such that $\psi^{-1}(a_{n_0})=a_{n_0}$ and $a_{n_0} \in \omega.$ But by Proposition \ref{prop2} $\psi$ is invertible at its fixed points which means that $a_{n_0}$ is a regular 
value of $\psi$ and
		\[\#\{\psi^{-1}(a_{n_0})\}=m \ge 2\] which is a contradiction! Hence the claim.
		\par\smallskip\no 
		Now by similar arguments as in the case of proper maps it follows that $\psi$ is a biholomorphism from $\omega$ to $B(0,R)$ and $p$ is a repelling fixed point of $\psi$ and hence lies in $J(\psi) \subset J(G).$ Since $\omega \cap J(G)=\{0\}$, we have $p=0$ which is an isolated point in the Julia set of $\psi$ and hence $\psi$ is conjugate to an upper triangular automorphism.
		\end{proof}
		\section{Recurrent and Wandering Fatou components of a semigroup G.}
		 As discussed in Section \ref{Introduction} we will be studying the properties of recurrent and wandering Fatou components of semigroup generated by entire maps of maximal generic rank on $\mathbb{C}^k$. The wandering and the recurrent Fatou components for a semigroup $G$ are defined as:
		 
		\begin{defn}
		Let $G=\langle\phi_1,\phi_2,\hdots \rangle$ where each $\phi_i\in \mathcal{E}_k $. Given a Fatou component $\Omega$ of $G$ and $\phi\in G$, let $\Omega_{\phi}$  be the Fatou component of $G$ containing $\phi(\Omega\setminus \Sigma_\phi)$ where $\Sigma_{\phi}$ is the set where the Jacobian of $\phi$ vanishes. A Fatou component is  {\it wandering} if the set $\big\{ \Omega_{\phi}:\phi\in G\big\}$ contains infinitely many distinct elements.
		\end{defn}
		
		\begin{defn}
		Let $G=\langle\phi_1,\phi_2,\hdots \rangle$ where each $\phi_i\in \mathcal{E}_k $. A Fatou component $\Omega$ of $G$ is {\it recurrent} if for any sequence $\{g_j\}_{j\geq 1}\subset G$, there exists a subsequence $\{g_{j_m}\}$ and a point $p\in \Omega$ (the point $p$ depends on the chosen sequence) such that $g_{j_m}(p)\rightarrow p_0 \in \Omega$.
		\end{defn}
		
		\no
		Note that we assume here a stronger definition of recurrence than the existing definition for the case of iterations of a single holomorphic endomorphism of $\mbb C^k.$ The natural extension of this definition to the semigroup set up would have been the following,  a Fatou component $\Omega$ is recurrent if there is a  point $p \in \Omega$ and a sequence $\phi_n \in \Omega$ such that $\phi_n(p) \to p_0,$ where $p_0 \in \Omega.$ If this definition of recurrence is adopted then it is possible that a {\it Recurrent} domain is{\it Wandering.}  In particular, Theorem 5.3 in \cite{HM1} gives an example of a polynomial semigroup $G=\langle \phi_1,\phi_2, \hdots \rangle$ in $\mbb C$, such that there exists a Fatou component, (say $\cal{B}$, which is conformally equivalent to a disc), that is wandering, but returns to the same component infinitely often. This means that there exists sequences say $\phi_n^+ \in G$ and $\phi_n^- \in G$ such that $\phi_n^-(\cal{B}) \subset \cal{B}$ or $\phi_n^+(\cal{B})$ are contained in distinct Fatou components of $G.$ This example can be easily adapted in higher dimensions.
		\begin{ex}
  Consider the semigroup $\cal{G}=\langle \Phi_1, \Phi_2, \hdots, \rangle$  generated by the maps
\[\Phi_i(z,w)=(\phi_i(z), w^2)\] where $\phi_i$ are the polynomial maps as in Theorem 5.3 of \cite{HM1}. Let $\{\Phi_n^-\}_{n \ge 1} \subset G$ be the sequence that maps $\cal{B} \times \mathbb{D}$ into itself and $\{\Phi_n^+\}_{n \ge 1} \subset G$ be the sequence such that $$\Phi^+_i(\cal{B}\times \mbb D) \cap \Phi^+_j( \cal{B}\times \mbb D) = \emptyset$$ for every $i \neq j.$ Also $\cal{B}\times \mbb D$ is a Fatou component of $\cal{G}$ as any point on the boundary of $\cal{B}\times \mbb D$, is either in the Julia set of $G$ or in the Julia set of the map $z \to z^2$. Hence $\cal{B}\times \mbb D$ is a Fatou component which is wandering, but may be recurring as well if we adapt the classical definition of recurrence.
		\end{ex}
		
		\no Hence we work with a stronger  definition of recurrence than the classical one. Next we provide an alternative description for recurrent Fatou components of $G$.
		\begin{lem}
		A Fatou component $\Omega$ is recurrent if and only if  for any sequence $\{\phi_j\}\subset G$, there exists a compact set $K\subset \Omega$ and a subsequence $\{\phi_{j_m}\}$  such that $\phi_{j_m}(p_{j_m})\rightarrow p_0 \in \Omega$ for a sequence $\{p_{j_m}\}\subset K$.
		\end{lem}
		\begin{proof}
		Take any sequence $\{\phi_j\}\subset G$. Then there exists a subsequence $\{\phi_{j_m}\}$ and points $\{p_{j_m}\}\subset K$ with $K$ compact in $\Omega$ such that $$\phi_{j_m}(p_{j_m})\rightarrow p_0 \in \Omega.$$ Without loss of generality we assume $p_{j_m}\rightarrow q_0 \in K$.
		It follows that $\phi_{j_m}(q_0)\rightarrow p_0 \in \Omega$ using the fact that any sequence of $G$ is normal on the Fatou set of $G$.
		\end{proof}

		\begin{prop}\label{recurrence_prop}
		  Let $G=\langle \phi_1,\phi_2,\hdots ,\phi_m\rangle$ where each $ \phi_i \in \mathcal{E}_k$ for every $1 \le i \le m$. If $\Omega$ is a recurrent Fatou component of $G$, then $G$ is locally bounded on $\Omega$. Moreover $\Omega$ is pseudoconvex and Runge.
		  \end{prop}
		  \begin{proof}
		  Assume $G$ is not locally bounded on $\Omega$. Then there exists a compact set $K\subset\Omega$  and $\{g_r\} \subseteq  G$  such that $\lvert g_r(z_r) \rvert > r$ with $z_r \in K$ for every $r \geq 1$. Clearly this can not be the case since $\Omega$ is a recurrent Fatou component, so we can always get a subsequence $\{g_{r_k}\}$ from the sequence $\{g_r\} \in G$ such that it converges to a holomorphic function uniformly on compact set in $\Omega$ and in particular on $K$. From the proof of Proposition \ref{bound}, it follows that local boundedness of $G$ on $\Omega$ implies that $\Omega$ is polynomially convex. Hence $\Omega$ is pseudoconvex.
		  \end{proof}
		  
	%
	\begin{thm}
	 Let $G= \langle \phi_1,\phi_2, \hdots \rangle$ where each $\phi_i \in \mathcal{E}_k.$ Assume that $\Omega$ is a {\it recurrent} Fatou component of $G.$ If there exists a $\phi \in G$ such that $\phi(\Omega)$ is contained in the Fatou set of $G$ i.e., $\phi(\Omega) \subset F(G)$ then one of the following is true
	 \begin{enumerate}
	  \item [(i)] There exists an attracting fixed point (say $p_0$) in $\Omega$ for the map $\phi.$
	  \item [(ii)] There exists a closed connected submanifold $M_\phi \subset \Omega$ of dimension $r_\phi$ with $1 \le r_\phi \le k-1$  and an integer $l_\phi >0$ such that 
	  \begin{enumerate}
	  \item[(a)] $\phi^{l_\phi}$ is an automorphism of $M_{\phi}$ and $\overline{\{\phi^{nl_\phi}\}_{n \ge 1}}$ is a compact subgroup of ${\rm Aut}(M_\phi).$
	  \item[(b)] If $f \in \overline{\{\phi^n\}}$, then $f$ has maximal generic rank $r_\phi$ in $\Omega.$
	  \end{enumerate}
	  \item[(iii)] $\phi$ is an automorphism of $\Omega$ and $\overline{\{\phi^n\}}$ is a compact subgroup of ${\rm Aut}(\Omega).$ 
	 \end{enumerate}

	\end{thm}
	\begin{proof}
	 Since $\Omega \subset F(G)$, there exists a recurrent Fatou component of the map $\phi$ (say $\Omega_{\phi}$) such that $\Omega \subset \Omega_{\phi},$ i.e., there exists an integer $l\geq1$ such that 
	 \[\phi^l(\Omega_\phi) \cap \Omega_{\phi}\neq \emptyset\;\; \text{and} \;\; \phi^{m}(\Omega_\phi) \cap \Omega_{\phi}= \emptyset \] 
	 for $0 \leq m < l. $ So, if $l>1$ then there do not exist any $p \in \Omega$ such that any subsequence of $\{\phi^{lk+1}(p)\}_{k\geq1}$ converges to a point in $\Omega$. Hence $l=1$ and by assumption it follows that $\phi(\Omega) \subset \Omega.$
	 
	 
	 
	 \medskip
	 
\no Let $h$ be a limit function of $\{\phi^n\}$ of maximal rank (say  $r_\phi$). i.e., 
\[h(p)=\lim_{j \to \infty} \phi^{n_j}(p) \;\; \text{for every}\;\; p \in \Omega,\]
where $\{n_j\}$ is an increasing subsequence of natural numbers.
	 
%

	 \medskip
	 
	\no {\it Case 1:} If $r_{\phi}=0.$ Then $h(\Omega)=p_0$ for some $p_0 \in \Omega$ since by recurrence there exists a point $p \in \Omega$, such that $\phi^{n_j}(p) \to p_0$ and $p_0 \in \Omega.$ Also $h(p_0)=p_0.$  Then
	 \[ \phi(p_0)= \phi(h(p_0))=h(\phi(p_0))=p_0,\]
	i.e., $p_0$ is a fixed point of $\phi.$ As some sequence of iterates of $\phi$ converge to a constant function, $p_0$ is an attracting fixed point for $\phi.$
	
	\medskip
	
	\no {\it Case 2:} If $r_{\phi} \geq 1.$ Then there exists an increasing subsequence $\{m_j\}$ such that
	 \[ p_j=m_{j+1}-m_j \]
	  are increasing positive integers  and the sequences $\{\phi^{m_j}\}$ and $\{\phi^{p_j}\}$ converge uniformly to the limit functions  ${h}$  and $\tilde{h}$ respectively on the Fatou component $\Omega.$ Since by recurrence $h(\Omega) \cap \Omega \neq \emptyset$, if $p \in \Omega$ be such that $p =h(q)$ for some $q \in \Omega$ then
	  \[\tilde{h}(p)= \lim_{j \to \infty} \phi^{m_{j+1}-m_j}(p)=\lim_{j \to \infty}\phi^{m_{j+1}-m_j}(\phi^{m_j}(q))=p\]
	  Define \[ M=\{x \in \Omega : \tilde{h}(x)=x \}.\]
	  {\it Claim:} $M$ is a closed complex submanifold of $\Omega.$ 
	  
	  \medskip
	  
	  \no Since $h(\Omega) \cap \Omega \subset M$, $M$ is a variety of dimension $\ge r_\phi$. But by the choice of $h$, the generic rank of $\tilde{h} \le r_{\phi}$ and $M \subset \tilde{h}(\Omega) \cap \Omega.$ So the dimension of $M$ is $r_{\phi}.$ Now for any point in $M,$ the rank of the derivative matrix of ${\rm Id}-\tilde{h}$ is greater than or equal to $k-r_\phi$. Suppose for some $x \in M$ the rank of  $D({\rm Id}-\tilde{h})(x)>k-r_\phi,$ then there exists a small neighbourhood of $x$, say $V_x$ such that $V_x \subset \Omega$ and \[\text{rank of }{\rm Id}-\tilde{h}>k-r_\phi \;\; \text{for every }\;\; x \in  V_x.\]Then $\{{\rm Id}-\tilde{h}\}^{-1}(0) \cap V_x$ is a variety of dimension at most $r_\phi-1$ i.e., the dimension of $M$ is strictly less than $r_\phi,$ which is a contradiction. Thus the rank of ${\rm Id}-\tilde{h}$ is $k-r_\phi$ for every point in $M$ and hence $M$ is a closed submanifold of $\Omega.$
	  


\medskip

\no {\it Step 1:} Suppose that $r_{\phi}=k.$

\medskip

\no Then clearly $M=\Omega$ and $\tilde{h}$ on $\Omega$ is the identity map. Let $h_2= \lim \phi^{p_j-1}.$ Then 
\[\tilde{h}(x)=h_2 \circ \phi(x)=x, \; \; \text{for every}\;\; x \in \Omega\]
i.e., $\phi$ is injective on $\Omega$ and $\phi(\Omega)$ is an open subset of $\Omega.$ Suppose there exists an $x \in \Omega \setminus \phi(\Omega)$ then for a sufficiently small ball of radius $r>0$ with $B_r(x) \subset \Omega$
\[
\phi^l(\Omega) \cap B_r(x)=\emptyset \;\;\text{for every}\;\; l \ge1.
\] 
This contradicts that $\phi^{p_j} (x)\to x.$ Hence $\phi$ is surjective on $\Omega$ and hence an automorphism of $\Omega.$

\medskip
%
%
%

%


\no {\it Step 2:} Suppose that $ 1 \le r_{\phi} \le k-1. $ Let $ M_\phi  $ denote an irreducible component of $ M. $ For every $ q \in M_{\phi}  $,  it follows that $ \phi^{p_j}(q) \to q $ as $j \to \infty.$ Since $  \phi(\Omega) \subset \Omega $, we get $\phi^n(q) \in \Omega$ for every $n \ge 1$ and 
\[ \tilde{h} \circ \phi^n(q)=\phi^n \circ \tilde{h}(q)=\phi^n(q) \; \text{for every}\; q \in M_{\phi},\]
i.e., $\phi^n(M_\phi) \subset M$ for every $n \ge 1.$ 

\medskip

\no{\it Claim:} There exists a positive integer $l_\phi$ such that $ \phi^{l_{\phi}}(M_{\phi}) \subset M_{\phi} .$ 

\medskip 

\no Let $p_0 \in M_\phi$ and $ \De \subset \Omega$ be a polydisk at $ p_0 $ such that $\De$ does not intersect the other components of $M_{\phi} . $ Now choose $\De' \subset \De$, a sufficiently small polydisk such that $\tilde{h}(\De') \subset \De.$ Then $\omega =\tilde{h}(\De') \subset M_\phi$ is a $r_\phi-$ dimensional manifold. Let $ \De'' $ be a $ r_\phi -$ dimensional polydisk inside $ \omega $ and   $ \{w_l\}_{l \ge 1}$ be a sequence in $\De''$ such that it converges to some $ w_0 \in  \De''. $ But $\phi^{p_j}(w_{p_j}) \to w_0$ as $j \to \infty$ hence \[ \phi^{p_j}(M_\phi) \cap \De \neq \emptyset,\; \text{i.e.,}\; \phi^{p_j} (M_\phi)\subset(M_\phi)  \] for $j$ sufficiently large. Let $l_\phi$ be the minimum value  such that $M_\phi$ is invariant under $\phi^{l_{\phi}}.$

\medskip

\no{\it Claim:} $ \phi^{l_\phi} $ is an automorphism of $ M_\phi. $

\medskip

\no Without loss of generality there exists a sequence $ \{k_j\} $ such that $p_j=i_0+k_jl_\phi$ for some $0 \le i_0 \le l_\phi -1$ i.e.,
\[ \phi^{i_0} \circ \phi^{k_jl_\phi}(x) \to x \;\; \text{for every} \;\; x \in M_{\phi}.\] As $ M_{\phi} $ is invariant under $ \phi^{l_\phi} $, the sequence $x_j=\phi^{k_jl_\phi}(x)$ lies in $ M_{\phi}. $ Again as before let $\De_x$ be a sufficiently small neighbourhood such that $\De_x \subset \Omega$ and $\De_x$ does not intersect the other components of $M.$ Since $\phi^{i_0}(x_j)  \in \De_x \cap M_\phi$ for large $j$, $\phi^{i_0}(M_\phi) \subset M_\phi.$ But $0 \le i_0 \le l_\phi-1$, i.e., $i_0=0$ and $\{\phi^{k_jl_\phi}\}$ converges uniformly to the identity on $M_{\phi}.$ Let $\psi=\lim \phi^{(k_j-1)l_\phi}$ then \[ \phi^{l_\phi} \circ \psi(x)=\psi \circ \phi^{l_\phi}(x)=x \;\; \text{for every}\;\; x \in M_\phi.\]
\no Hence $\phi^{l_\phi}$ is injective on $M_\phi$ and $\phi^{l_\phi}(M_\phi)$ is an open subset in the manifold $M_\phi$. Now as in {\it Step 1} observe that $\phi^{k_jl_\phi}$ converges to the identity on $M_\phi$ for an unbounded sequence $ \{k_j\} $, so $\phi^{l_\phi}$ is also surjective on $ M_{\phi} $. Thus the claim.

\medskip

\no Let $Y=\{\phi^{nl_\phi}\}_{n \ge 1} \subset {\rm Aut}(M_\phi).$

\medskip

\no {\it Claim:} $\bar{Y}$ is a locally compact subgroup of ${\rm Aut}(M_\phi).$ 

\medskip 

\no For some $\Psi \in Y$ and for a compact set $K \subset M_\phi$ consider the neighbourhood of $\Psi$ given by
\[V_\Psi (K, \ep)= \{\psi \in {\rm Aut}(M_\phi): \|\psi(z)-\Psi(z)\|_{K} < \ep \}.\]
One can choose $\ep$ and $K$ sufficiently small such that for every  sequence $\psi_j \in V_\Psi (K, \ep)$ there exists an open set $U \subset \Omega$ such that $\psi_j(U \cap M_\phi) \subset \bar{V}  \cap M_\phi \subset \Omega$, where $V$ is some open subset of $\Omega.$

\medskip

\no Since $\psi_j=\phi^{n_j l_\phi}$ for a sequence $\{n_k\}$ and $\Omega$ is a Fatou component, $\psi_j$ has a convergent subsequence in $\Omega.$ We choose appropriate subsequences such that the  limit maps $$ \Psi_1=\lim_{j \to \infty } \phi^{n_jl_{\phi}} \;\;\text{and}\;\; \Psi_2=\lim_{j \to \infty} \phi^{(k_j-n_j)l_\phi} $$ is defined on $\Omega.$ Also as $M_\phi$ is closed in $\Omega$,  $\Psi_i(M_\phi) \subset \overline{M_\phi}$ for every $i=1,2$ where $ \overline{M_\phi} $ denote the closure of $ M_\phi $ in $ \mbb C^k. $
Then $\Psi_1(U)  \subset \Omega$ and
\begin{align}\label{injective}
\Psi_2 \circ \Psi_1(x)=x \;\; \text{for every}\;\; x \in U \cap M_\phi.
\end{align}  
Since $\Psi_1$ on $M_\phi$ is a limit of automorphisms of $M_\phi$, the  Jacobian of $\Psi_1$ on the manifold $M_\phi$ is either non-zero  at every point of $M_\phi$ or vanishes identically. But by (\ref{injective}),  $\Psi_1$ restricted to $U \cap M_{\phi}$ is injective, which is open in the manifold $M_\phi$ i.e., $\Psi_1$  is an open map of $M_\phi$ and $\Psi_1(M_\phi) \subset M_{\phi}.$ So  (\ref{injective}) is true for every $x \in M_\phi.$  Now by the same arguments it follows that  $\Psi_2$
is an injective map from $M_\phi$ such that $\Psi_2(M_\phi) \subset M_{\phi}.$ Hence
\[\Psi_2 \circ \Psi_1(x)=\Psi_1 \circ \Psi_2(x)=x \;\; \text{for every}\;\; x \in M_\phi,\]
i.e. $\Psi_1$ is an automorphism of $M_\phi.$ This proves that $\bar{Y}$ is a locally compact  subgroup of ${\rm Aut(M_\phi)}.$

\medskip

\no Now since $M_\phi$ is a complex manifold and $\bar{Y}$ is a locally abelian subgroup of automorphisms of $M_\phi$, by Theorem A in \cite{Bochner}, it follows that $\bar{Y}$ is a Lie group. Hence the component of $\bar{Y}$ containing the identity is isomorphic to $\mbb T^l \times \mbb R^m.$ Suppose $\Psi$ is the isomorphism,  then for some $n >0$, $\Psi(a,b)=\phi^{nl_\phi}$. Now if $b \neq 0$, then there does not exist an increasing sequence of $k_j$ such that $\phi^{k_j l_\phi}$ converges to identity. This proves that the component of $\bar{Y}$ containing the identity is compact and hence any component of $\bar{Y}$ is compact by the same arguments. Also as $M_\phi$ is contained in the Fatou set, the number of components of $\bar{Y}$ is finite, thus $\bar{Y}$ is a compact subgroup of ${\rm Aut}(M_\phi).$

\medskip

\no If $r_\phi=k$, then $M_\phi$ is $\Omega$, then one can apply the same technique as discussed above to conclude that $\overline{\{\phi^n\}}$ is a closed compact subgroup of ${\rm Aut}(\Omega).$

\medskip

\no Finally, let $f$ be a limit of $\{\phi^n\}_{n \ge 1}$ i.e., \[f(p)=\lim_{j \to \infty} \phi^{n_j}(p) \;\; \text{for every} \;\; p \in \Omega.\]

\no {\it Claim:} The generic rank of $f$ is $r_{\phi}.$

\medskip

\no By the definition of recurrence it follows that $\Omega \subset \Omega_{\phi}$, where $\Omega_\phi$ is a periodic Fatou component for $\phi$ with period $1.$ Hence by Theorem 3.3 in \cite{FS} it follows that the limit maps of the set $\{\phi^n\}$ in $\Omega_\phi$ have the same generic rank (say $r$). But $\Omega$ is an open subset of the Fatou component $\Omega_\phi$, so the rank of limit maps restricted to $\Omega$ should be same, i.e., $r=r_\phi$ and each limit map of $\{\phi^n\}$ has rank $r_\phi.$ 
\end{proof}

\no By Proposition \ref{recurrence_prop} a semigroup $G$ is always locally uniformly  bounded on a recurrent Fatou component semigroup $G$. If $G$ is finitely generated by holomorphic endomorphisms of maximal rank $k$ in $\mbb C^k$, then by Proposition \ref{bound} it follows that a recurrent Fatou component is mapped in the Fatou set by any elemnet of $G.$ Hence we have the following corollary.
\begin{cor}
	 Let $G= \langle \phi_1,\phi_2, \hdots,\phi_m \rangle$ where each $\phi_i \in \mathcal{E}_k$ for every $1 \le i \le m.$ Assume that $\Omega$ is a {\it recurrent} Fatou component of $G$ then for every $\phi \in G$ one of the following is true
	 \begin{enumerate}
	  \item [(i)] There exists an attracting fixed point (say $p_0$) in $\Omega$ for the map $\phi.$
	  \item [(ii)] There exists a closed connected submanifold $M_\phi \subset \Omega$ of dimension $r_\phi$ with $1 \le r_\phi \le k-1$  and an integer $l_\phi >0$ such that 
	  \begin{enumerate}
	  \item[(a)] $\phi^{l_\phi}$ is an automorphism of $M_{\phi}$ and $\overline{\{\phi^{nl_\phi}\}_{n \ge 1}}$ is a compact subgroup of ${\rm Aut}(M_\phi).$
	  \item[(b)] If $f \in \overline{\{\phi^n\}}$, then $f$ has maximal generic rank $r_\phi$ in $\Omega.$
	  \end{enumerate}
	  \item[(iii)] $\phi$ is an automorphism of $\Omega$ and $\overline{\{\phi^n\}}$ is a compact subgroup of ${\rm Aut}(\Omega).$ 
	 \end{enumerate}

	\end{cor}

		\begin{ex}
		Let $G=\langle \phi_1,\phi_2 \rangle$ be a semigroup of entire maps in $\mbb C^2$  generated by 
		\[\phi_1(z,w)=(w, \alpha z-w^2) \; \; \mbox{and} \;\; \phi_2(z,w)=(zw,w)\]
		where $0<\alpha < 1.$ Then $G$ is locally uniformly bounded on a sufficiently small neighbourhood around the origin, and $\phi(0)=0$ for every $\phi \in G.$ So the Fatou component of $G$ containing $0$ (say $\Omega_0$) is recurrent. Now note that for $\phi_2$ 
		 $$r_{{\phi_2}}=1\;\; \mbox{ and} \;\;M_{\phi_2}=\{(0,w): w \in \mbb C\} \cap \Omega_0,$$ whereas for $\phi_1$ the origin is an attracting fixed point. This illustrates the different behaviour of the sequences $\{\phi_1^n\}$ and $\{\phi_2^n\}$ (both of which are in $G$) on $\Omega_0.$
		 
		 \medskip
		 
		\no Note that for any other $\phi \in G$ which is not of the form $\phi_1^k,~k \ge 2$, contains a factor of $\phi_2$ at least  once. Since for a small enough ball (say $B$) around origin, $\phi_2$ is contracting, and $\phi_1(B) \subset B$
		so there exists  a constant $0<a_\phi<1$ such that $$|\phi(z)| \le a_{\phi} |z| \;\; \text{for every} \;\; z \in B,$$i.e., the origin is an attracting fixed point.  
	  		\end{ex}
		\begin{prop}
		Let $G=\langle \phi_1,\phi_2,\hdots ,\phi_m\rangle$ where each $\phi_i\in \mathcal{V}_k$ for every $1 \le i \le m$ and let $\Omega$ be an invariant Fatou component of $G$. Then either $\Omega$ is recurrent or there exists a sequence $\{\phi_n\}\subset G$ converging to infinity.
		\end{prop}
		\begin{proof}
		If $\Omega$ is not recurrent, then there exists a sequence $\{\phi_n\}\subset G$ such that $\{\phi_n\} \ra \partial \Omega \cup \{\infty\} $ uniformly on compact sets of $\Omega$. Assume $\{\phi_{n_k}\}$ converges to a holomorphic function $f$ on $\Omega$ . This implies that $f(\Omega)\subset \partial\Omega$ contradicting the assumption that each $\phi_{n_k}$ is volume preserving. Hence $\{\phi_{n_k}\}$ diverges to infinity uniformly on compact subsets of $\Omega$.
		\end{proof}
		\begin{prop}
		Let $G=\langle \phi_1,\phi_2,\hdots ,\phi_m \rangle$ where each $\phi_i\in \mathcal{V}_k$ for every $1 \le i \le m$ and let $\Omega$ be a wandering Fatou component of $G$. Then there exists a sequence $\{\phi_n\}\subset G$ converging to infinity.
		\end{prop}
		\begin{proof}
		Since $\Omega$ is wandering, one can choose a sequence $\{\phi_n \}\subset G$ so that
		\begin{equation}
		 \Omega_{\phi_n}\cap  {\Omega}_{\phi_m}=\emptyset \label{0}
		\end{equation}
		for $n\neq m$. If this sequence $\{\phi_n\}$ does not diverge to infinity uniformly on compact subsets, some subsequence $\{\phi_{n_k}\}$ will converge to a holomorphic function $h$ on $\Omega$. By abuse of notation, we denote $\{\phi_{n_k}\}$ still by $\{\phi_n \}$. Fix $z_0\in \Omega$. Then for any given $\epsilon$, there exists $\delta$ such that
		\begin{equation}
		\left|\phi_{n_0}(z)-\phi_n(z)\right|<\epsilon \label{1}
		\end{equation}
		for all $n\geq n_0$ and for all $z\in B(z_0,\delta)$. From (\ref{1}) it follows that vol$(\cup_{n\geq n_o}\phi_n(B(z_0,\delta)))$ is finite. On the other hand, since each ${\phi_n}$ is volume preserving and (\ref{0}) holds, we get \[\mbox{Vol}\Big(\bigcup_{n\geq n_o}\phi_n\big(B(z_0,\delta)\big)\Big) =+\infty.\] Hence we have proved the existence of  a sequence in $G$ converging to infinity.
		\end{proof}
		
		\section{Concluding Remarks} 
		As mentioned in the introduction, the classification of recurrent Fatou components for iterations of holomorphic endomorphisms of complex projective spaces  has been studied in \cite{FS2} and \cite{Fornaess-Rong}. It would be interesting to explore the same question for semigroups of holomorphic endomorphisms of complex projective spaces.  The main theorem in \cite{FS2} and \cite{Fornaess-Rong} are proved under the assumption that the given recurrent Fatou component is also forward invariant. The analogue of such a condition in the case of semigroups is not clear to us since we are then dealing with a family of maps none of which is distinguishable from the other. 
	\bibliographystyle{amsplain}
	\bibliography{ref}
	\end{document}